\documentclass[a4paper,10pt,twoside,notitlepage]{amsart}

\usepackage[T1]{fontenc}
\usepackage[utf8]{inputenc}
\usepackage{lmodern}
\usepackage{amsmath,amssymb,amsthm,mathtools,mathrsfs}
\usepackage[left=28mm,right=28mm,top=26mm,bottom=26mm]{geometry}
\usepackage[expansion=false]{microtype}
\usepackage[unicode,hidelinks]{hyperref}
\usepackage{enumitem}
\usepackage{comment}

\theoremstyle{plain}
\newtheorem{thm}{Theorem}[section]
\newtheorem{prop}[thm]{Proposition}
\newtheorem{lem}[thm]{Lemma}
\newtheorem{cor}[thm]{Corollary}
\theoremstyle{definition}
\newtheorem{definition}[thm]{Definition}
\theoremstyle{remark}
\newtheorem{rmk}[thm]{Remark}
\numberwithin{equation}{section}

\newcommand{\R}{\mathbb R}

\newcommand{\p}{\partial}
\newcommand{\Mi}{M^{\mathrm{int}}}
\newcommand{\Mji}[1]{M_{#1}^{\mathrm{int}}}
\newcommand{\wt}{\widetilde}
\newcommand{\Id}{\operatorname{Id}}

\newcommand{\Hess}{\operatorname{Hess}}
\newcommand{\tr}{\operatorname{tr}}
\newcommand{\grad}{\operatorname{grad}}
\newcommand{\divg}{\operatorname{div}}

\newcommand{\tf}{\mathrm{tf}}
\newcommand{\norm}[1]{\lVert #1\rVert}

\newcommand{\pair}[2]{\langle #1,#2\rangle}
\newcommand{\dd}{\,\mathrm d}
\newcommand{\Egam}{C_c^\infty(\Gamma;\R)}
\newcommand{\para}[1]{\vspace{3mm}\noindent\textbf{#1.}}
\setlist[enumerate]{itemsep=3pt,topsep=5pt}
\title[Conformal Poisson embedding method]{The conformal Poisson embedding method for the anisotropic Calder\'on problem with a potential}
\author[C. I. C\^arstea]{C\u{a}t\u{a}lin I. C\^arstea}
\address[C. I. C\^arstea]{Department of Applied Mathematics, National Yang Ming Chiao Tung University, Hsinchu 30050, Taiwan}
\email{catalin.carstea@gmail.com}
\author[T. Liimatainen]{Tony Liimatainen}
\address[T. Liimatainen]{Department of Mathematics and Statistics, University of Jyv\"askyl\"a, Jyv\"askyl\"a, Finland}
\email{tony.t.liimatainen@jyu.fi}
\author[Y.-H. Lin]{Yi-Hsuan Lin}
\address[Y.-H. Lin]{Department of Applied Mathematics, National Yang Ming Chiao Tung University, Hsinchu, Taiwan}
\curraddr{Fakult\"at f\"ur Mathematik, Universit\"at Duisburg-Essen, Essen, Germany}
\email{yihsuanlin3@gmail.com}
\author[W. Trad]{William Trad}
\address[W. Trad]{Department of Mathematics and Statistics, University of Jyv\"askyl\"a, Jyv\"askyl\"a, Finland}
\email{tradwing@jyu.fi}
\date{}
\hypersetup{
 pdftitle={The conformal Poisson embedding method for the anisotropic Calderon problem with a potential},
 pdfauthor={C\u{a}t\u{a}lin I. C\^arstea; Tony Liimatainen; Yi-Hsuan Lin; William Trad},
 pdfsubject={Revised manuscript},
 pdfkeywords={Calderon problem, Schrodinger equation, partial data, Poisson embedding, conformal gauge}
}

\begin{document}
\begin{abstract}
We prove that the local Dirichlet-to-Neumann map of $\Delta_g+q$ determines a compact connected real-analytic Riemannian manifold of dimension at least three and a real-analytic real-valued potential, up to a diffeomorphism fixing the observed boundary and the conformal gauge of the equation. The result is new even when \(q\) is constant, i.e. for the fixed-frequency (Helmholtz) equation \(\Delta_g u+\lambda u=0\). The measurements are restricted to an arbitrary nonempty open boundary set $\Gamma$. The proof develops conformal  Poisson embeddings method and uses coordinates given by quotients of solutions.  



\end{abstract}
\maketitle

	\section*{Disclaimer}
	This is a preliminary version, released rapidly because of the urgent pace of AI-driven developments in the field. The main results, arguments, and statements of this paper were developed without the use of AI. AI tools were used for auxiliary mathematical discussions and for assistance with parts of the writing. The authors take full responsibility for the final mathematical statements and proofs. The authors will improve the presentation in an updated version following this initial release.

\section{Introduction}\label{sec:introduction}

We consider the inverse problem of recovering a real analytic Riemannian metric $g$ and a real analytic potential $q$ from the Dirichlet-to-Neumann map of $\Delta_g+q$. The measurements are restricted to an arbitrary nonempty open subset $\Gamma$ of the boundary. We prove uniqueness up to a diffeomorphism fixing $\Gamma$ and the conformal gauge of the equation.

\subsection{The boundary value problem}
Let $(M,g)$ be a compact connected Riemannian manifold with nonempty smooth boundary, of dimension $n\geq3$, and let $q\in C^\infty(M;\R)$. We use the nonnegative sign convention for the Laplace--Beltrami operator:
\begin{equation}\label{eq:laplace_convention}
 \Delta_g u=-\divg_g\grad_g u
 =-|g|^{-1/2}\p_a\bigl(|g|^{1/2}g^{ab}\p_bu\bigr).
\end{equation}
Thus $\Delta_g u=-\tr_g\Hess_g u$. Throughout the paper we assume that zero is not a Dirichlet eigenvalue of
\[
 L_{g,q}:=\Delta_g+q.
\]
Equivalently, $L_{g,q}:H^2(M)\cap H^1_0(M)\to L^2(M)$ is invertible. By the Fredholm alternative and elliptic regularity, for every $f\in C^\infty(\p M;\R)$ the Dirichlet problem
\begin{equation}\label{eq:dirichlet}
 \begin{cases}
 L_{g,q}u_f=0&\text{in }\Mi,\\
 u_f=f&\text{on }\p M
 \end{cases}
\end{equation}
has a unique solution $u_f\in C^\infty(M;\R)$. Here and below smoothness on a manifold with boundary means smoothness up to the boundary. The Dirichlet-to-Neumann map is
\[
 \Lambda_{g,q}f=\p_{\nu_g}u_f\big|_{\p M},
\]
where $\nu_g$ is the outward unit normal. The same map is bounded from $H^{1/2}(\p M)$ to $H^{-1/2}(\p M)$.

For a nonempty open subset $\Gamma\subset\p M$, define the local Dirichlet-to-Neumann map by
\begin{equation}\label{eq:local_DN}
 \Lambda_{g,q}^{\Gamma}f
 =\p_{\nu_g}u_f\big|_\Gamma,
 \quad f\in\Egam.
\end{equation}
The boundary datum is extended by zero outside $\Gamma$ before solving \eqref{eq:dirichlet}. Thus both the Dirichlet inputs and the Neumann observations are restricted to the same set $\Gamma$. No regularity of the relative boundary of $\Gamma$ is assumed.

\subsection{The gauge and the main result}
The equation $L_{g,q}u=0$ has a conformal gauge in addition to the usual diffeomorphism gauge; see, for example, \cite[Proposition~8.2]{DosSantosFerreira2009}. We set
\begin{equation}\label{eq:p_exponent}
 p=\frac{2n}{n-2},\quad p-2=\frac4{n-2},\quad p-1=\frac{n+2}{n-2}.
\end{equation}
For a positive smooth function $c$, define
\begin{equation}\label{eq:gauge}
 \wt g=c^{p-2}g,
 \quad
 \wt q=c^{1-p}(\Delta_gc+qc).
\end{equation}
Then
\begin{equation}\label{eq:covariance_intro}
 L_{\wt g,\wt q}(c^{-1}u)=c^{1-p}L_{g,q}u.
\end{equation}
In particular, if $c|_\Gamma=1$ and $\p_{\nu_g}c|_\Gamma=0$, this transformation leaves the local DN map unchanged. A diffeomorphism of $M$ that fixes $\Gamma$ pointwise also leaves these measurements unchanged. The verification, including the boundary normalization, is given in Proposition~\ref{hidden_conformal_invariance}.

For a manifold with an identified boundary portion $\Gamma$, write
\[
 M^\Gamma=\Mi\cup\Gamma.
\]
When two manifolds have mutual boundary, we mean that an identification of their boundaries has been fixed. In the analytic setting this identification is assumed real analytic. All boundary functions and local DN maps are compared using this identification.

\begin{thm}\label{thm:main}
Let $(M_j,g_j)$, $j=1,2$, be compact connected real analytic Riemannian manifolds of the same dimension $n\geq3$, with mutual real analytic boundary. Let $q_j\in C^\omega(M_j;\R)$, and assume that zero is not a Dirichlet eigenvalue of $L_{g_j,q_j}$. Real analyticity of $g_j$ and $q_j$ is assumed up to the boundary. Let $\Gamma$ be a nonempty open subset of the mutual boundary. If
\begin{equation}\label{eq:DN_equality}
 \Lambda_{g_1,q_1}^{\Gamma}f=\Lambda_{g_2,q_2}^{\Gamma}f
 \quad\text{for every } f\in\Egam,
\end{equation}
then there are a real analytic diffeomorphism
\[
 J:M_1^\Gamma\longrightarrow M_2^\Gamma
\]
and a positive function $c\in C^\omega(M_1^\Gamma)$ such that
\begin{equation}\label{eq:boundary_normalization}
 J|_\Gamma=\Id,
 \quad c|_\Gamma=1,
 \quad \p_{\nu_{g_1}}c|_\Gamma=0,
\end{equation}
and
\begin{equation}\label{eq:main_gauge}
 J^*g_2=c^{p-2}g_1,
 \quad
 J^*q_2=c^{1-p}(\Delta_{g_1}c+q_1c).
\end{equation}
Moreover, the corresponding Dirichlet solutions satisfy
\begin{equation}\label{eq:main_solution_relation}
 u_f^1=c\,J^*u_f^2
 \quad\text{on }M_1^\Gamma,
 \quad f\in\Egam.
\end{equation}
The pair $(J,c)$ with these properties is unique.
\end{thm}

Here a function or map is real analytic on $M_j^\Gamma$ if it is real analytic in the interior and in analytic boundary charts at every point of $\Gamma$. The theorem does not assert that the prescribed identification of the unobserved boundary is fixed by $J$.

\begin{cor}[Full boundary data]\label{cor:full_data}
Under the assumptions of Theorem~\ref{thm:main}, equality of the full DN maps implies that \eqref{eq:boundary_normalization}--\eqref{eq:main_gauge} hold for a real analytic diffeomorphism $J:M_1\to M_2$ and a positive $c\in C^\omega(M_1)$, with the boundary conditions imposed on all of $\p M_1$. Conversely, such a diffeomorphism and conformal gauge give the same full DN map.
\end{cor}

Neither the interior manifold nor its conformal class is assumed known. No simplicity or product-structure assumption is made, and the result does not provide a stability estimate.

When $q_1=q_2=0$, the second identity in \eqref{eq:main_gauge} gives $\Delta_{g_1}c=0$. The Cauchy data of $c-1$ vanish on $\Gamma$, so unique continuation implies $c=1$. In this case the conclusion reduces to the usual isometric identification for the analytic anisotropic Calder\'on problem. For independent potentials, the conformal factor cannot in general be removed.

\subsection{Relation to earlier work}

Calder\'on's inverse conductivity problem asks whether the boundary voltage-to-current map determines an electrical conductivity~\cite{Calderon1980}. Sylvester and Uhlmann~\cite{SylvesterUhlmann1987} proved global uniqueness for smooth isotropic conductivities in dimensions at least three by constructing complex geometrical optics solutions. For anisotropic conductivities, changes of coordinates that fix the boundary leave the data unchanged. In dimensions $n\geq3$, the geometric formulation gives the corresponding problem for the Laplace--Beltrami operator.

For real analytic metrics, Lee and Uhlmann~\cite{Lee-Uhlmann} proved boundary determination and global uniqueness under additional topological assumptions. Lassas and Uhlmann~\cite{Lassas-Uhlmann} recovered the analytic manifold itself without these restrictions. Their proof uses Green functions to identify points and continue a known boundary neighborhood. Lassas, Taylor, and Uhlmann~\cite{LassasTaylorUhlmann2003} extended the method to complete real analytic manifolds, with measurements on an arbitrary open boundary subset. Guillarmou and S\'a Barreto~\cite{GuillarmouSaBarreto2009} obtained local-data uniqueness for compact connected Einstein manifolds using the Einstein equation.

Our proof follows the Poisson embedding approach of Lassas, Liimatainen, and Salo~\cite{lassas2018poisson}. The Poisson embedding identifies a point by the values there of all harmonic functions with prescribed boundary data. Harmonic coordinates continue a local identification, and Runge approximation supplies the solution jets used to recover the metric. The same paper treats more general elliptic equations with interior source-to-solution measurements. For the boundary problem with an unknown potential, we allow corresponding solution values to differ by a common positive factor.

A related problem is the inverse problem for the conformal Laplacian
\[
 \Delta_g+\frac{n-2}{4(n-1)}\operatorname{Scal}_g,
\]
where $\operatorname{Scal}_g$ is the scalar curvature. Lassas, Liimatainen, and Salo~\cite{lassas2016calder} recovered the conformal geometry from its full DN map for locally conformally real analytic manifolds. Their proof uses coordinates formed from quotients of solutions. Lassas and Liimatainen~\cite{lassas2019conformal} developed these conformal harmonic coordinates further and applied them to the regularity of conformal maps. We use the same quotient construction for $\Delta_g+q$. Here $q$ is an independent unknown and the measurements may be restricted to $\Gamma$. We assume that the metric itself is analytic, whereas~\cite{lassas2016calder} allows a smooth metric locally conformal to an analytic one.

On smooth manifolds with suitable geometry, complex geometrical optics methods give other uniqueness results. Dos Santos Ferreira, Kenig, Salo, and Uhlmann~\cite{DosSantosFerreira2009} used limiting Carleman weights to recover potentials on a fixed admissible background and metrics in a prescribed conformal class. They also described the joint metric--potential gauge. The boundary argument below uses only their local symbol calculation after the potential has been made zero; it does not use their conformal normalization. Dos Santos Ferreira, Kurylev, Lassas, and Salo~\cite{DosSantosFerreira2016} extended this framework to transversally anisotropic geometries using ray transform information. These results use a specified background geometry; our analyticity assumption allows the manifold and its conformal class to remain unknown.

For partial measurements on a known Euclidean background in dimensions at least three, Bukhgeim and Uhlmann~\cite{BukhgeimUhlmann2002} and Kenig, Sj\"ostrand, and Uhlmann~\cite{KenigSjostrandUhlmann2007} recovered potentials using Carleman estimates. Isakov~\cite{Isakov2007local} proved uniqueness when the inaccessible boundary lies in a plane or a sphere. Kenig and Salo~\cite{KenigSalo2013} obtained partial-data results on manifolds using Carleman estimates, reflection, and ray transforms. The input and observation sets in these results, or the inaccessible boundary, satisfy restrictions determined by the geometry. For disjoint input and observation sets, Daud\'e, Kamran, and Nicoleau~\cite{DaudeKamranNicoleau2019} constructed anisotropic nonuniqueness examples. In Theorem~\ref{thm:main}, both inputs and observations lie in the same arbitrary open set $\Gamma$.

In dimension two, Nachman~\cite{Nachman1996} proved uniqueness for smooth planar isotropic conductivities, and Astala and P\"aiv\"arinta~\cite{AstalaPaivarinta2006} treated bounded measurable conductivities. Astala, P\"aiv\"arinta, and Lassas~\cite{AstalaLassasPaivarinta2005} recovered planar anisotropic conductivities up to changes of coordinates. For potentials on a fixed background, Bukhgeim~\cite{Bukhgeim2008} proved full-data uniqueness in the plane. Imanuvilov, Uhlmann, and Yamamoto~\cite{ImanuvilovUhlmannYamamoto2010} proved uniqueness from an arbitrary open boundary subset in planar domains, and Guillarmou and Tzou~\cite{GuillarmouTzou2011} obtained the analogous result on Riemann surfaces.

C\^arstea, Liimatainen, and Tzou~\cite[Theorem~1.1]{CarsteaLiimatainenTzou2024} proved simultaneous recovery of a smooth Riemannian surface and an independent potential from the full DN map, without an a priori identification of the surfaces or their topology. Their conclusion is a boundary-fixing diffeomorphism $F$ and a positive smooth function $\lambda$, with $\lambda|_{\p M_1}=1$, such that
\[
 g_1=\lambda F^*g_2,\quad q_1=\lambda^{-1}F^*q_2.
\]
Their proof recovers the boundary traces of holomorphic functions and then the conformal structure. This is the smooth two-dimensional counterpart of the simultaneous problem considered here; its conformal gauge differs from \eqref{eq:gauge} in dimensions $n\geq3$.

The approximation results used here follow the classical Runge arguments of Lax~\cite{Lax1956} and Malgrange~\cite{Malgrange1956}, see also \cite{lassas2018poisson}. For boundary values supported in $\Gamma$, duality reduces approximation to unique continuation from $\Gamma$ in the complement of the approximation region. 
We use qualitative approximation and its finite-dimensional consequence: exact realization of compatible second-order jets at finitely many interior points.

A preprint in the OpenAI mathematical manuscript collection~\cite{OpenAI2026} announces smooth metric-only uniqueness from one boundary patch. This result and the other related result~\cite{OpenAI2026b} in that collection are based in part on the Poisson embedding~\cite{lassas2018poisson}. Their equations do not contain a zeroth-order term, and hence do not exhibit the conformal gauge of the problem addressed in this paper, which is treated here by the conformal Poisson embedding method. The preprints mentioned above are not used in this paper, and our results do not follow from them.

\subsection{Outline of the proof}

The Poisson embedding associates to a point $x$ the functional
\[
 P(x)[f]=u_f(x),\quad f\in\Egam.
\]
We seek a diffeomorphism $J$ and a positive function $c$ such that
\[
 P_1(x)=c(x)P_2(J(x)).
\]
The factor $c$ cancels in coordinates of the form
\[
 Z=(u_{f_1},\ldots,u_{f_n})/u_{f_0},
\]
where $u_{f_0}$ does not vanish. We choose the denominator separately in each coordinate neighborhood, so no globally positive solution is needed.

Near a point of $\Gamma$, choose a boundary datum supported in $\Gamma$ and equal to one near that point. Its solutions are locally positive and analytic. Using them in~\eqref{eq:gauge} removes the potentials near the boundary point. Boundary determination gives equality of the gauged metric jets in boundary normal coordinates. Analyticity gives equality of these metrics in a neighborhood, and Cauchy uniqueness gives the initial scaled correspondence between solutions.

Runge approximation shows that solution values separate points even up to a scalar, and that compatible second-order jets can be prescribed. Comparing the two operators on these jets gives the metric and potential relations up to~\eqref{eq:gauge} and a constant metric scaling. The boundary correspondence fixes this constant. This local coefficient calculation precedes global continuation because it is also needed to exclude boundary limits.

Suppose that a sequence tends to an interior point while its images tend to the unobserved boundary. We choose a solution with a simple zero at the target boundary point. The solution ratios show that the corresponding source solution is nonzero at the interior limit. The metric $|u|^{p-2}g$ then has bounded sectional curvature on the source side, but its curvature tends to minus infinity on planes tangent to the level sets of $u$ on the target side. The local homothety obtained from coefficient recovery gives a contradiction. No bound on the multiplier is assumed.

We can therefore continue the correspondence by an open-and-closed argument in $Z$-coordinates. Applying the same construction to the inverse local correspondence proves surjectivity. The local boundary identifications then extend $J$ and $c$ to every point of $\Gamma$.

\para{Organization of the paper} Section~\ref{sec:boundary} proves boundary determination. Section~\ref{sec_poisson_embed} develops the Poisson embedding and $Z$-coordinates. Section~\ref{sec_recovery_of_metric} contains the local coefficient calculation, and Section~\ref{sec:global} proves global continuation. The Runge and prescribed-jet results are proved in Appendix~\ref{runge_apprx_sec}.

\section{Boundary determination}\label{sec:boundary}

\subsection{Conformal covariance}
\begin{prop}\label{hidden_conformal_invariance}
Let $c>0$ be smooth, and let $(\wt g,\wt q)$ be given by \eqref{eq:gauge}. Then \eqref{eq:covariance_intro} holds for every smooth $u$. Multiplication by $c^{-1}$ gives a bijection between the homogeneous Dirichlet kernels of $L_{g,q}$ and $L_{\wt g,\wt q}$.

If $c|_\Gamma=1$, then, for every $f\in\Egam$,
\begin{equation}\label{eq:DN_gauge}
 \Lambda_{\wt g,\wt q}^{\Gamma}f
 =\Lambda_{g,q}^{\Gamma}f-f\,\p_{\nu_g}c\big|_\Gamma.
\end{equation}
Consequently, the additional condition $\p_{\nu_g}c|_\Gamma=0$ makes the local DN maps equal. Pullback of both coefficients by a diffeomorphism fixing $\Gamma$ pointwise also preserves the local DN map.
\end{prop}
\begin{proof}
For a positive function $\lambda$, the convention \eqref{eq:laplace_convention} gives
\begin{equation}\label{eq:conformal_laplacian}
 \Delta_{\lambda g}v
 =\lambda^{-1}\Big(\Delta_gv-\frac{n-2}{2}
       \pair{d\log\lambda}{dv}_g\Big).
\end{equation}
The product rule is
\[
 \Delta_g(av)=a\Delta_gv+v\Delta_ga-2\pair{da}{dv}_g.
\]
Substitution of $\lambda=c^{p-2}$ and $v=c^{-1}u$ into \eqref{eq:conformal_laplacian} gives
\begin{equation}\label{eq:hidden_covariance}
 \Delta_{c^{p-2}g}(c^{-1}u)
 =c^{1-p}\Delta_gu-u c^{-p}\Delta_gc.
\end{equation}
Adding $\wt q c^{-1}u$ proves \eqref{eq:covariance_intro}. The statement about homogeneous Dirichlet kernels follows immediately, also in the Sobolev spaces defining the Dirichlet realizations.

If $c=1$ on $\Gamma$, the function $c^{-1}u_f$ has boundary value $f$ on the whole boundary: outside $\Gamma$ both boundary values are zero. The outward unit normal of $c^{p-2}g$ is $c^{-(p-2)/2}\nu_g$, so it equals $\nu_g$ on $\Gamma$. Therefore
\[
 \p_{\nu_{\wt g}}(c^{-1}u_f)
 =\p_{\nu_g}u_f-f\,\p_{\nu_g}c
 \quad\text{on }\Gamma,
\]
which proves \eqref{eq:DN_gauge}. Finally, naturality of the Laplace--Beltrami operator and of the unit normal under pullback gives the diffeomorphism invariance. A diffeomorphism fixing $\Gamma$ maps its boundary complement to itself, so the zero extension of a datum supported in $\Gamma$ is respected.
\end{proof}

\subsection{Boundary determination}
We use $a_1\simeq a_2$ to denote equality of all Taylor coefficients in specified coordinates. We first choose analytic gauges that make the potentials vanish near an observed boundary point. We then recover the metric jets in the boundary normal coordinates of the gauged metrics. We also note our conformal gauge is analytic near boundary, which later on will be used to have analytic boundary normal coordinates for conformal metrics.

\begin{lem}\label{lemma_boundarydetermination}
Assume the hypotheses of Theorem~\ref{thm:main}, including \eqref{eq:DN_equality}, and fix $x_0\in\Gamma$. There are positive functions $c_j\in C^\infty(M_j)$, which are analytic near the bonudary, such that
\begin{equation}\label{eq:local_gauge_boundary}
 c_j|_{\p M_j}=1,
 \quad
 \p_{\nu_{g_1}}c_1|_\Gamma=\p_{\nu_{g_2}}c_2|_\Gamma,
\end{equation}
and such that the coefficients
\[
 (\wt g_j,\wt q_j)
 =\bigl(c_j^{p-2}g_j,\ c_j^{1-p}(\Delta_{g_j}c_j+q_jc_j)\bigr)
\]
have equal local DN maps on $\Gamma$. Near $x_0$, the functions $c_j$ are real analytic and solve $L_{g_j,q_j}c_j=0$. In particular, $\wt q_j=0$ near $x_0$, and the gauged metrics are real analytic there. Their Taylor series agree in their respective boundary normal coordinates based on the same analytic coordinates on the mutual boundary.
\end{lem}
\begin{proof}
Let us first construct the functions $c_j$, $j=1,2$. Choose open boundary neighborhoods $\Sigma_0\Subset\Sigma_1\Subset\Gamma$ of $x_0$ and a function $f_0\in\Egam$ that equals one on a neighborhood of $\overline{\Sigma_1}$. Let
\[
 w_j=u_{f_0}^j.
\]
By continuity, $w_j>0$ in a sufficiently small collar of $\overline{\Sigma_1}$. Analytic elliptic boundary regularity applies locally because the coefficients, boundary, and boundary value $1$ are analytic there; see \cite[Theorem~A]{MorreyNirenberg1957}. Thus $w_j$ is analytic in that collar after it is made smaller. Since $w_j$ need not be positive on all of $M_j$, we modify it outside a neighborhood of $x_0$.

Choose $0\leq\chi_0\leq1$ smooth on the mutual boundary, supported in $\Sigma_1$ and equal to one near $\overline{\Sigma_0}$. Extend it to cutoffs $\chi_j$ supported in the collars where $w_j>0$, with $0\leq\chi_j\leq1$, $\chi_j|_{\p M_j}=\chi_0$, and $\p_{\nu_{g_j}}\chi_j|_{\p M_j}=0$. Such extensions are obtained by using the same tangential cutoff in each boundary collar and a normal cutoff that is one near the boundary. Set
\[
 c_j=1+\chi_j(w_j-1).
\]
These functions are positive, equal to $w_j$ near $x_0$, and equal to one on the whole boundary. Their normal derivatives are
\[
 \p_{\nu_{g_j}}c_j=\chi_0\Lambda_{g_j,q_j}f_0
 \quad\text{on }\p M_j.
\]
Consequently \eqref{eq:local_gauge_boundary} follows from the local DN equality. For every $f\in\Egam$, Proposition~\ref{hidden_conformal_invariance} gives
\[
 \Lambda_{\wt g_j,\wt q_j}^\Gamma f
 =\Lambda_{g_j,q_j}^\Gamma f-f\,\p_{\nu_{g_j}}c_j|_\Gamma.
\]
Hence
\begin{equation}\label{eq:gauged_DN_equality}
 \Lambda_{\wt g_1,\wt q_1}^\Gamma
 =\Lambda_{\wt g_2,\wt q_2}^\Gamma.
\end{equation}
The same proposition shows that the gauged Dirichlet realizations remain invertible.

Near $x_0$, the functions $c_j$ solve $L_{g_j,q_j}c_j=0$, so
\begin{equation}\label{eq:zero_gauged_potential}
 \wt q_j=c_j^{1-p}(\Delta_{g_j}c_j+q_jc_j)=0.
\end{equation}
Let us consider the gauged coefficients in their respective $\wt g_j$-boundary normal coordinates $(x',t)$, based on the same boundary coordinates, with $t$ the inward distance. We denote by Greek letters the indices $1,\ldots,n-1$. In these coordinates, put
\begin{equation}\label{eq:k}
 \wt k_j^{\alpha\beta}
 =\p_t\wt g_j^{\alpha\beta}
    -\tau_j\wt g_j^{\alpha\beta},
 \quad
 \tau_j=(\wt g_j)_{\gamma\delta}\p_t\wt g_j^{\gamma\delta},
\end{equation}
and
\begin{equation}\label{eq:l}
 \wt\ell_j^{\alpha\beta}
 =\frac14\p_t\wt k_j^{\alpha\beta}
       +\wt q_j\wt g_j^{\alpha\beta}
 =\frac14\p_t\wt k_j^{\alpha\beta}.
\end{equation}
Here $(\wt g_j^{\alpha\beta})$ is the inverse of $((\wt g_j)_{\alpha\beta})$. The last equality in \eqref{eq:l}, which holds on a neighborhood, is the reason for choosing the gauges from solutions.

We use the zero-potential part of the local boundary-symbol calculation in the proof of \cite[Lemma~8.7]{DosSantosFerreira2009}. We specify the order-by-order information needed here without imposing a determinant normalization. For one of the gauged metrics, write
\[
 s=(\wt g^{\alpha\beta}\xi_\alpha\xi_\beta)^{1/2},
 \quad \omega=\xi'/s.
\]
The local DN symbol is $-b|_{t=0}$ modulo a smoothing symbol, where $b\sim b_1+b_0+b_{-1}+\cdots$ and
\begin{equation}\label{eq:boundary_symbol_k}
 b_1=-s,\quad
 b_0=-\frac14\wt k^{\alpha\beta}\omega_\alpha\omega_\beta+R_0.
\end{equation}
For $m\geq1$, the same symbol recursion gives
\begin{equation}\label{eq:boundary_symbol_ell}
 b_{-m}
 =-\frac{(\p_t^{m-1}\wt\ell^{\alpha\beta})
                    \omega_\alpha\omega_\beta}{(2s)^m}+R_m.
\end{equation}
The term $R_0$ depends only on the boundary metric and its tangential derivatives. The term $R_m$ depends only on normal metric derivatives of orders at most $m$ and their tangential derivatives. Indeed, in the recursion the only term introducing the next normal derivative is $-(2b_1)^{-1}\p_t$ applied to the preceding symbol. Since the potential vanishes near the boundary, it introduces no additional unknown derivatives. These observations give \eqref{eq:boundary_symbol_ell} inductively from \eqref{eq:boundary_symbol_k}. Equality~\eqref{eq:gauged_DN_equality}, localized by cutoffs supported in $\Gamma$, gives equality of the homogeneous DN symbols on the boundary patch.

We prove by induction that
\[
 \p_t^L\wt g_1^{\alpha\beta}
 =\p_t^L\wt g_2^{\alpha\beta}\quad\text{on }t=0,
 \quad L\geq0.
\]
The claim for $L=0$ follows from the principal symbol. Assume it holds through order $L$. If $L=0$, equality of $b_0$ in \eqref{eq:boundary_symbol_k} determines $\wt k_1=\wt k_2$. If $L\geq1$, the remainders $R_L$ agree by the induction hypothesis, so \eqref{eq:boundary_symbol_ell} determines equality of $\p_t^{L-1}\wt\ell_j$, and hence of $\p_t^L\wt k_j$. In both cases the lower derivatives of $\wt k_j$ also agree, by \eqref{eq:k} and the already known metric jets. Thus
\begin{equation}\label{eq:k_derivs}
 \p_t^r\wt k_1^{\alpha\beta}
 =\p_t^r\wt k_2^{\alpha\beta},\quad 0\leq r\leq L.
\end{equation}
Contracting \eqref{eq:k} gives
\begin{equation}\label{eq:k_trace}
 (\wt g_j)_{\alpha\beta}\wt k_j^{\alpha\beta}
 =-(n-2)\tau_j.
\end{equation}
The derivatives through order $L$ of the left-hand side agree for $j=1,2$, since matrix inversion preserves equality of the known metric jets. Because $n\geq3$, the derivatives $\p_t^r\tau_j$ also agree for $0\leq r\leq L$. Differentiating \eqref{eq:k} now gives
\begin{equation}\label{eq:metric_jet_induction}
 \p_t^{L+1}\wt g_j^{\alpha\beta}
 =\p_t^L\wt k_j^{\alpha\beta}
   +\sum_{r=0}^L\binom Lr
      (\p_t^r\tau_j)(\p_t^{L-r}\wt g_j^{\alpha\beta}).
\end{equation}
All terms on the right-hand side agree. This proves the induction step. Tangential differentiation and matrix inversion show that the full Taylor series of the gauged metrics agree. The metrics and the gauges are analytic near $x_0$ by their construction, and the gauged potentials vanish there.
\end{proof}

\begin{rmk}\label{rmk:boundary_normalization_choice}
The boundary normal coordinates above are constructed after the conformal change, using the gauged metrics. This argument does not use the determinant normalization in \cite[Lemma~8.5]{DosSantosFerreira2009}. The problem with that normalization argument, and a corrected construction, are discussed in \cite[Section~2]{lassas2016calder}. Here the local equation $L_{g_j,q_j}c_j=0$ replaces that normalization, and only the zero-potential symbol calculation is needed. For the next lemma, recall the $c_j$ is analytic near the boundary.
\end{rmk}

\begin{lem}[Determination near the boundary]\label{lemma_determination_near_the_boundary}
Under the assumptions of Theorem~\ref{thm:main}, for every $x_0\in\Gamma$ there are connected neighborhoods $U_j\subset M_j^\Gamma$ of $x_0$, a real analytic diffeomorphism $F:U_1\to U_2$, and a positive $c_F\in C^\omega(U_1)$ such that
\begin{gather}\label{eq:local_boundary_identification}
 F|_{U_1\cap\Gamma}=\Id,
 \quad F^*g_2=c_F^{p-2}g_1,\\
 \label{eq:local_boundary_multiplier}
 c_F|_{U_1\cap\Gamma}=1,
 \quad\p_{\nu_{g_1}}c_F|_{U_1\cap\Gamma}=0,
\end{gather}
and, for every $f\in\Egam$,
\begin{equation}\label{eq:local_scaled_solutions}
 u_f^1=c_F\,(u_f^2\circ F)\quad\text{on }U_1.
\end{equation}
\end{lem}
\begin{proof}
Let $c_j$ and $\wt g_j$ be as in Lemma~\ref{lemma_boundarydetermination}. Choose the same analytic coordinates on the boundary near $x_0$, and let $\psi_j$ be the boundary normal coordinates for $\wt g_j$ based on those coordinates. These charts are analytic, since the metric, the boundary normal field, and the geodesic equation are analytic. Equality of all jets and convergence of the Taylor series imply, after shrinking to a common coordinate half-neighborhood, that
\[
 (\psi_1^{-1})^*\wt g_1=(\psi_2^{-1})^*\wt g_2.
\]
Consequently $F=\psi_2^{-1}\circ\psi_1$ is an analytic local isometry for the gauged metrics and fixes the boundary portion.

Fix $f\in\Egam$. Near $x_0$, the functions $c_j^{-1}u_f^j$ are harmonic for $\wt g_j$. In the common coordinates their Dirichlet values are both $f$, and their normal derivatives agree by \eqref{eq:DN_gauge} and \eqref{eq:local_gauge_boundary}. The difference of their coordinate representations solves the same elliptic equation with zero Cauchy data on an open boundary patch. Boundary unique continuation, for instance \cite[Section~3.3]{Isakov_book}, gives
\[
 c_1^{-1}u_f^1=(c_2^{-1}u_f^2)\circ F
\]
in a fixed connected neighborhood. The neighborhood is chosen from the equality of the coefficients and does not depend on $f$: unique continuation extends the equality throughout that neighborhood. Set
\begin{equation}\label{eq:local_c_formula}
 c_F=\frac{c_1}{c_2\circ F}.
\end{equation}
This gives \eqref{eq:local_scaled_solutions}. The equality
$F^*(c_2^{p-2}g_2)=c_1^{p-2}g_1$ gives the metric identity in \eqref{eq:local_boundary_identification}.

On the boundary $c_j=1$. Since $F$ is an isometry for the gauged metrics and these metrics have the same unit normals as the original metrics on the boundary, $DF(\nu_{g_1})=\nu_{g_2}$. Differentiating \eqref{eq:local_c_formula} there and using \eqref{eq:local_gauge_boundary} gives \eqref{eq:local_boundary_multiplier}.
\end{proof}

\section{Poisson embedding}\label{sec_poisson_embed}
Throughout this section, $M$ is compact and connected, $g$ and $q$ are smooth and real, and the Dirichlet realization of $L_{g,q}$ is invertible. Real analyticity is required only where stated. Boundary values are supported in a fixed nonempty open set $\Gamma\subset\p M$. When two manifolds are compared, their boundaries are identified and $\Gamma$ denotes the same observed set.

\subsection{Definition and regularity}
\begin{definition}[Poisson embedding]\label{Poisson_embed_Def}
The Poisson embedding associated with $(M,g,q)$ and $\Gamma$ is
\[
 P:M^\Gamma\longrightarrow\mathcal D'(\Gamma),
 \quad P(x)[f]=u_f(x),\quad f\in\Egam.
\]
Here $\mathcal D'(\Gamma)$ is the space of real distributions, the continuous dual of $C_c^\infty(\Gamma;\R)$.
\end{definition}

The same formula defines a continuous map on all of $M$, with value zero on $\p M\setminus\Gamma$. At an observed boundary point $x\in\Gamma$, $P(x)=\delta_x$.
We will show that solution values separate points even up to a scalar and that quotients of solutions give local coordinates. These properties will allow us to study correspondences of the form $P_1=c\,P_2\circ J$.

\begin{prop}\label{P_basics}
The map $P$ is smooth, and its differential is given by
\begin{equation}\label{formula_for_the_differential}
 (DP_xX)[f]=du_f(x)[X].
\end{equation}
More precisely, let $k\geq0$ and $s>k+(n-1)/2$, and let $\mathcal E_s(\Gamma)$ be the closure of $C_c^\infty(\Gamma)$ in $H^s(\p M)$, where functions are extended by zero. The map defined by $P$ on all of $M$ is $C^k$ with values in $\mathcal E_s(\Gamma)'$. If $g$ and $q$ are real analytic in the interior, every scalar function $x\mapsto P(x)[f]$ is real analytic there.
\end{prop}
\begin{proof}
Elliptic Dirichlet regularity and the nonresonance assumption give
\[
 \norm{u_f}_{H^{s+1/2}(M)}\leq C_s\norm{f}_{H^s(\p M)}.
\]
Choose $\alpha>0$ sufficiently small that $s+1/2>k+\alpha+n/2$. Sobolev embedding bounds $u_f$ in $C^{k,\alpha}(M)$ by the right-hand side. Taylor's formula, uniformly over boundary data in the unit ball of $\mathcal E_s(\Gamma)$, proves the asserted Banach-valued differentiability and \eqref{formula_for_the_differential}. These estimates for all $k$ also give smoothness with values in distributions. Analytic interior regularity~\cite{MorreyNirenberg1957} proves the last assertion. We do not require analytic boundary data for interior analyticity.
\end{proof}

\subsection{Runge approximation and prescribed jets}
The main tool is Runge approximation, which allows one to approximate locally defined solutions by global solutions with boundary values supported in $\Gamma$. We record the consequences that will be used in the proof.  

\begin{prop}[Runge approximation]\label{prop_runge_approximation_cauchy}
Let $O\Subset\Mi$ be an open set with smooth boundary, possibly with finitely many connected components, and suppose that $M\setminus\overline O$ is connected. If $v$ solves $L_{g,q}v=0$ in a neighborhood of $\overline O$, then there are $f_m\in\Egam$ such that
\[
 u_{f_m}\longrightarrow v\quad\text{in }L^2(O).
\]
The convergence also holds in $C^k(K)$ for every compact $K\subset O$ and every nonnegative integer $k$.
\end{prop}
This is Proposition~A.2 in \cite{lassas2018poisson}, applied to \(L=L_{g,q}\) and \(U=O\). The proof there applies verbatim to an open set with finitely many connected components, since the only connectedness used is that of \(M\setminus\overline O\). The \(H^s(O)\) density given there for every \(s\ge1\) yields the \(L^2\) and \(C^k\) conclusions by Sobolev embedding.

A proof of the following proposition is given in Appendix~\ref{runge_apprx_sec}.
\begin{prop}[Prescribed second-order jets]\label{runge_prescribed_jet}
Let $x_1,\ldots,x_N$ be distinct points of $\Mi$. For each $\ell$, prescribe a scalar $a_\ell\in\R$, a covector $\xi_\ell\in T_{x_\ell}^*M$, and a symmetric covariant two-tensor $H_\ell$ satisfying
\begin{equation}\label{eq:jet_constraint}
 \tr_g H_\ell=q(x_\ell)a_\ell.
\end{equation}
There exists $f\in\Egam$ such that
\[
 u_f(x_\ell)=a_\ell,\quad
 du_f(x_\ell)=\xi_\ell,\quad
 \Hess_g u_f|_{x_\ell}=H_\ell
 \quad (1\leq\ell\leq N).
\]
Equivalently, the Hessian may be prescribed in the form
\begin{equation}\label{eq:jet_tf}
 H_\ell=H_{\ell,0}+\frac{q(x_\ell)a_\ell}{n}g_{x_\ell},
 \quad \tr_gH_{\ell,0}=0.
\end{equation}
\end{prop}

The condition \eqref{eq:jet_constraint} is exactly the equation $L_{g,q}u_f(x_\ell)=0$ under the sign convention \eqref{eq:laplace_convention}. There are no restrictions on the value or the first derivative. The proposition remains valid with $\Gamma$ replaced by any smaller nonempty open boundary set.

The prescribed jets imply the following separation property.

\begin{prop}[Separation up to a scalar]\label{prop:projective_separation}
For each $x\in M^\Gamma$, $P(x)\ne0$. If $x,y\in M^\Gamma$ and $a,b\in\R\setminus\{0\}$ satisfy
\[
 aP(x)=bP(y),
\]
then $x=y$ and $a=b$. In particular, $P$ is injective and is a topological embedding onto its image.
\end{prop}
\begin{proof}
At an interior point, Proposition~\ref{runge_prescribed_jet} supplies a solution with value one. At a boundary point in $\Gamma$, choose a boundary datum with value one there. Thus $P(x)\ne0$.

Suppose first that $x\ne y$ are interior points. The prescribed-jet proposition gives a single boundary datum with $u_f(x)=0$ and $u_f(y)=1$, contradicting the displayed equality. If both points belong to $\Gamma$, their Dirac distributions are not proportional. If $x\in\Gamma$ and $y\in\Mi$, choose a nonempty open set $\Sigma\Subset\Gamma\setminus\{x\}$. Applying Proposition~\ref{runge_prescribed_jet} with boundary values supported in $\Sigma$ gives $u_f(y)=1$ while $u_f(x)=f(x)=0$, again a contradiction. Interchanging $x$ and $y$ handles the other mixed case. Once $x=y$, nonvanishing of $P(x)$ implies $a=b$.

To verify the topological assertion, extend $P$ to all of the compact manifold $M$ as above. If $x\in M^\Gamma$ and $W$ is an open neighborhood of $x$ in $M$, then $P(x)$ does not belong to $P(M\setminus W)$: separation applies at all points of $M^\Gamma$, and the remaining points have image zero. The set $P(M\setminus W)$ is compact, hence closed in the Hausdorff space of distributions. Its complement is a neighborhood of $P(x)$ whose inverse image lies in $W$. This proves continuity of the inverse on $P(M^\Gamma)$.
\end{proof}

\subsection{Nonzero normal derivatives}
We also need a solution whose boundary value is supported away from a given boundary point and whose normal derivative is nonzero there.

\begin{lem}\label{lem:remote_normal}
Let $y\in\p M$ and let $\Sigma\subset\p M$ be nonempty and open, with $y\notin\Sigma$. Then the functional
\[
 N_y:C_c^\infty(\Sigma;\R)\longrightarrow\R,
 \quad N_yf=\p_{\nu_g}u_f(y),
\]
is not zero. In particular, for some such $f$ the solution $u_f$ vanishes on a boundary neighborhood of $y$ and has $du_f(y)\ne0$.
\end{lem}
\begin{proof}
Let $G(x,z)$ be the Dirichlet Green kernel of $L_{g,q}$ relative to the Riemannian volume measure. It exists by invertibility, is real and symmetric, and is smooth up to the boundary away from its diagonal. The Poisson representation is
\begin{equation}\label{eq:poisson_kernel_representation}
 u_f(x)=\int_{\p M}K(x,z)f(z)\dd S_g(z),
 \quad K(x,z)=-\p_{\nu_z}G(x,z).
\end{equation}
These are the usual Green-kernel consequences of elliptic Dirichlet regularity; they apply to an invertible real Schr\"odinger operator without a sign assumption on $q$.

For the fixed point $y$, set $K_y(x)=-\p_{\nu_y}G(x,y)$ for $x\in\Mi$. This function solves $L_{g,q}K_y=0$ in $\Mi$ and extends smoothly to each boundary patch not containing $y$, with zero Dirichlet value there. Symmetry of $G$ and differentiation of \eqref{eq:poisson_kernel_representation} off the diagonal give
\[
 N_yf=\int_\Sigma\p_{\nu_z}K_y(z)f(z)\dd S_g(z).
\]
If $N_y$ were zero, $K_y$ would have zero Dirichlet and Neumann data on $\Sigma$. Boundary and interior unique continuation would then give $K_y=0$ throughout the connected interior.

To see explicitly that $K_y$ cannot vanish identically, choose $w\in C^\infty(M)$ with $w|_{\p M}=0$ and $\p_{\nu_g}w(y)\ne0$. A boundary defining function times a cutoff provides such a $w$. Fix $r>n$, and choose $F_m\in C_c^\infty(\Mi)$ converging to $L_{g,q}w$ in $L^r(M)$. For the zero-Dirichlet solutions $w_m=L_{g,q,D}^{-1}F_m$, elliptic estimates and invertibility give
\[
 w_m\longrightarrow w\quad\text{in }W^{2,r}(M),
 \quad\text{hence in }C^1(M).
\]
On the other hand, symmetry of the Green kernel yields
\[
 \p_{\nu_g}w_m(y)=-\int_M K_y(z)F_m(z)\dd V_g(z).
\]
If $K_y$ were identically zero, these derivatives would all be zero, contradicting their convergence to $\p_{\nu_g}w(y)\ne0$. This proves the lemma. The last assertion follows because $f$ has compact support in $\Sigma$, hence vanishes near $y$, and its nonzero normal derivative is a component of $du_f(y)$.
\end{proof}

\subsection{\texorpdfstring{$Z$}{Z}-coordinates and scaled identifications}
We use the term $Z$-coordinates from the conformal Laplacian construction in~\cite{lassas2016calder}, developed further in~\cite{lassas2019conformal} under the name conformal harmonic coordinates. The definition here applies the same quotient construction to $L_{g,q}$ with a general potential.
\begin{definition}\label{def:Z_coords}
Local coordinates $(Z^1,\ldots,Z^n)$ on a neighborhood $U$ are called $Z$-coordinates for $L_{g,q}$ if
\begin{equation}\label{Z_form1}
 Z^k=\frac{u^k}{u^0},\quad L_{g,q}u^k=L_{g,q}u^0=0,
 \quad u^0>0\text{ on }U.
\end{equation}
At a boundary point, the coordinates are required to be a boundary chart. A denominator that is nonzero at a point may always be made positive on a smaller neighborhood by changing its sign.
\end{definition}

\begin{lem}\label{lem:Z_existence}
Every point of $M^\Gamma$ has $Z$-coordinates formed from global Dirichlet solutions whose boundary values belong to $\Egam$. At an interior point, the denominator can be any prescribed global solution that is nonzero there, after a possible change of sign. If the coefficients are analytic in the interior, these coordinates are analytic at interior points. If the coefficients and the boundary are analytic near a point of $\Gamma$, the coordinates can be chosen analytic up to that boundary point.
\end{lem}
\begin{proof}
Let $x\in\Mi$, and let $v=u_{f_0}$ with $v(x)\ne0$. Given any covector $\xi\in T_x^*M$, Proposition~\ref{runge_prescribed_jet} supplies $u_f$ with
\[
 u_f(x)=0,\quad du_f(x)=v(x)\xi.
\]
It follows that $d(u_f/v)(x)=\xi$. Choosing a basis of covectors and applying the inverse function theorem gives $Z$-coordinates. Analytic interior regularity gives the analytic assertion.

Now let $x\in\Gamma$, and choose boundary coordinates $y^1,\ldots,y^{n-1}$ near $x$, centered at $x$. Choose boundary data $f_0,f_1,\ldots,f_{n-1}$ supported in $\Gamma$ such that, on a neighborhood of $x$,
\[
 f_0=1,\quad f_\alpha=y^\alpha.
\]
Choose a nonempty open $\Sigma\Subset\Gamma\setminus\{x\}$. By Lemma~\ref{lem:remote_normal}, there is $f_n\in C_c^\infty(\Sigma)$ for which $\p_{\nu_g}u_{f_n}(x)\ne0$. The function $u_{f_n}$ is zero on a boundary neighborhood of $x$. The ratios $u_{f_\alpha}/u_{f_0}$ restrict to $y^\alpha$ on the boundary, while $u_{f_n}/u_{f_0}$ vanishes there and has nonzero normal derivative. Their differentials are linearly independent. After changing the sign of the last numerator, its ratio is positive on the interior side of a small neighborhood, so the inverse function theorem gives a boundary chart. The denominator is positive near $x$ by continuity. In the analytic case, choose analytic boundary coordinates $y^\alpha$; all these boundary data are locally analytic at $x$, and analytic boundary regularity~\cite[Theorem~A]{MorreyNirenberg1957} gives the conclusion.
\end{proof}

The differential $DP_x$ is injective at every $x\in M^\Gamma$. If $DP_xX=0$, then $du_f(x)[X]=0$ for every $f\in\Egam$ by~\eqref{formula_for_the_differential}. All differentials of the quotient coordinates constructed in Lemma~\ref{lem:Z_existence} therefore vanish on $X$, so $X=0$. Together with Propositions~\ref{P_basics} and~\ref{prop:projective_separation}, this proves that $P$ is a smooth embedding.

We will repeatedly use the following consequence of the interior coordinate construction:
\begin{equation}\label{eq:quotient_differentials}
 \left\{d(u_f/v)(x):f\in\Egam\right\}=T_x^*M
 \quad\text{if }x\in\Mi\text{ and }v(x)\ne0.
\end{equation}

\begin{lem}[Regularity of a scaled identification]\label{J_is_local_dif}
Let $B\subset\Mji1$ be open. Suppose that $J:B\to\Mji2$ is continuous and that $c:B\to(0,\infty)$ satisfies
\begin{equation}\label{p1_p2_relation}
 u_f^1(x)=c(x)u_f^2(J(x))
 \quad (x\in B,\ f\in\Egam).
\end{equation}
Then $J$ is a smooth diffeomorphism onto its image and $c$ is smooth. If both pairs of coefficients are analytic in the interior, $J$ and $c$ are analytic.
\end{lem}
\begin{proof}
Fix $x\in B$ and choose $Z$-coordinates
\[
 U=(u_{f_1}^2,\ldots,u_{f_n}^2)/u_{f_0}^2
\]
near $J(x)$. The denominator $u_{f_0}^1$ is nonzero at $x$ by \eqref{p1_p2_relation}. By continuity of $J$, on a small neighborhood of $x$ its image stays in the coordinate neighborhood. Set
\[
 V=(u_{f_1}^1,\ldots,u_{f_n}^1)/u_{f_0}^1.
\]
The multiplier cancels, so $V=U\circ J$ and
\begin{equation}\label{representation_for_J}
 J=U^{-1}\circ V.
\end{equation}
This proves the asserted regularity of $J$, and
\[
 c=\frac{u_{f_0}^1}{u_{f_0}^2\circ J}
\]
proves the regularity of $c$.

For every $f$ we also have
\begin{equation}\label{equation_of_injectivity}
 d\left(\frac{u_f^1}{u_{f_0}^1}\right)_x[X]
 =d\left(\frac{u_f^2}{u_{f_0}^2}\right)_{J(x)}[DJ_xX].
\end{equation}
If $DJ_xX=0$, equation~\eqref{eq:quotient_differentials} implies $X=0$. Thus $DJ_x$ is invertible. Finally, if $J(x)=J(y)$, \eqref{p1_p2_relation} makes $P_1(x)$ and $P_1(y)$ proportional. Proposition~\ref{prop:projective_separation} gives $x=y$. The inverse function theorem now proves that $J$ is a diffeomorphism onto its open image.
\end{proof}

\section{Recovery of the metric and potential from scaled solutions}\label{sec_recovery_of_metric}
In this section we show that a local scaled identification of Dirichlet solutions determines the coefficients. This is the Schr\"odinger counterpart of the metric recovery from harmonic functions. We use the full compatible second-order jet rather than only trace-free Hessians used in \cite{lassas2018poisson}.

\begin{prop}\label{prop:operator_recovery}
Let $(M_j,g_j,q_j)$ satisfy the smooth assumptions of Section~\ref{sec_poisson_embed}. Let $B\subset\Mji1$ be connected and open, let $J:B\to\Mji2$ be a smooth diffeomorphism onto its image, and let $c>0$ be smooth on $B$. Suppose that \eqref{p1_p2_relation} holds for every $f\in\Egam$. Then there is a constant $\kappa>0$ such that
\begin{equation}\label{eq:local_coefficient_recovery}
 J^*g_2=\kappa c^{p-2}g_1,
 \quad
 J^*q_2=\kappa^{-1}c^{1-p}(\Delta_{g_1}c+q_1c)
 \quad\text{on }B.
\end{equation}
More generally, for every locally defined smooth function $v$ on $J(B)$,
\begin{equation}\label{eq:operator_intertwining}
 L_{g_1,q_1}(c\,J^*v)
 =\kappa c^{p-1}J^*(L_{g_2,q_2}v).
\end{equation}
If $J^*g_2=c^{p-2}g_1$ at one point of $B$, then $\kappa=1$.
\end{prop}
\begin{proof}
Write $h=J^*g_2$ and $r=J^*q_2$. The functions $v_f=u_f^2\circ J$ satisfy $L_{h,r}v_f=0$. By the hypothesis, $L_{g_1,q_1}(cv_f)=0$. Define an operator acting on functions on $B$ by
\begin{equation}\label{eq:conjugated_operator}
 Av=c^{-1}L_{g_1,q_1}(cv)
 =\Delta_{g_1}v-2\pair{d\log c}{dv}_{g_1}
    +\bigl(c^{-1}L_{g_1,q_1}c\bigr)v.
\end{equation}
Thus $Av_f=0$ for every $f\in\Egam$.

Fix $x\in B$. Evaluation of $L_{h,r}$ at $x$ is a nonzero linear functional on the finite-dimensional space of second-order jets of functions at $x$. Proposition~\ref{runge_prescribed_jet} on $M_2$, transported by the local diffeomorphism $J$, says that the jets of the functions $v_f$ fill exactly the kernel of this functional. Since evaluation of $A$ vanishes on that kernel, there is a scalar $\mu(x)$ such that, as functionals on all second-order jets,
\[
 A|_x=\mu(x)L_{h,r}|_x.
\]
Comparing the principal coefficients gives
\[
 g_1^{ab}=\mu h^{ab},\quad h=\mu g_1.
\]
Positive definiteness gives $\mu>0$, and, for example,
$\mu=n^{-1}\tr_{g_1}h$ shows that $\mu$ is smooth. Since the jet identity holds at every point, it is the differential-operator identity
\begin{equation}\label{eq:A_mu_L}
 A=\mu L_{h,r}\quad\text{on }B.
\end{equation}

Using \eqref{eq:conformal_laplacian} with $h=\mu g_1$, the right-hand side of \eqref{eq:A_mu_L} is
\[
 \Delta_{g_1}v-\frac{n-2}{2}\pair{d\log\mu}{dv}_{g_1}
                  +\mu r v.
\]
Comparison with \eqref{eq:conjugated_operator} gives
\[
 2d\log c=\frac{n-2}{2}d\log\mu,
 \quad c^{-1}L_{g_1,q_1}c=\mu r.
\]
Connectedness of $B$ therefore implies
\[
 \mu=\kappa c^{4/(n-2)}=\kappa c^{p-2}
\]
for one constant $\kappa>0$. This proves \eqref{eq:local_coefficient_recovery}. Multiplying \eqref{eq:A_mu_L} by $c$ and using naturality under pullback proves \eqref{eq:operator_intertwining}. The final assertion follows from the metric identity.
\end{proof}

\begin{cor}[Quotients of local solutions]\label{conf_preserves_Z}
Under the assumptions of Proposition~\ref{prop:operator_recovery}, let $V\subset J(B)$ be open. The map $v\mapsto c\,J^*v$ is a bijection between smooth solutions of $L_{g_2,q_2}v=0$ on $V$ and smooth solutions of $L_{g_1,q_1}u=0$ on $J^{-1}(V)$. In particular, if $v,w$ solve $L_{g_2,q_2}v=L_{g_2,q_2}w=0$ on $V$ and $w>0$, then
\[
 J^*\left(\frac vw\right)
 =\frac{cJ^*v}{cJ^*w}\quad\text{on }J^{-1}(V),
\]
where both numerator and denominator solve $L_{g_1,q_1}u=0$ and the denominator is positive. Thus $J$ preserves quotients of local solutions and pulls back $Z$-coordinates on $V$ to $Z$-coordinates on $J^{-1}(V)$.
\end{cor}
\begin{proof}
Apply~\eqref{eq:operator_intertwining} on $J^{-1}(V)$. Since $J:J^{-1}(V)\to V$ is a diffeomorphism and the factors $c$ and $\kappa c^{p-1}$ never vanish, it identifies the solution spaces in both directions. The quotient identity follows by cancellation of $c$.
\end{proof}

\begin{rmk}\label{rmk:solution_metrics}
An immediate consequence, needed in the next section, is the following. For a fixed datum $f$, on a region where the corresponding solutions do not vanish, put
\[
 \widehat g_j=|u_f^j|^{p-2}g_j.
\]
Then \eqref{p1_p2_relation} and \eqref{eq:local_coefficient_recovery} give
\begin{equation}\label{eq:solution_metric_homothety}
 J^*\widehat g_2=\kappa\widehat g_1.
\end{equation}
The constant $\kappa$ is the same one as in Proposition~\ref{prop:operator_recovery}; it does not depend on $f$.
\end{rmk}

\section{Global continuation}\label{sec:global}
We use $Z$-coordinates and real analyticity to extend the local scaled identification to the whole interior. We first exclude limits at the unobserved boundary, where the multiplier might become unbounded. Lemma~\ref{lem:no_boundary_escape} is used in the closedness step of Theorem~\ref{thm:global_scaled}; its last paragraph contains the curvature calculation needed when the limiting point is unobserved.

\subsection{Exclusion of boundary limits}
\begin{lem}\label{lem:no_boundary_escape}
Let $(M_j,g_j,q_j)$ satisfy the smooth assumptions of Section~\ref{sec_poisson_embed}. Let $B\subset\Mji1$ be connected and open, let $J:B\to\Mji2$ be a smooth diffeomorphism onto its image, and let $c>0$ be smooth on $B$. Suppose that
\[
 u_f^1=c\,u_f^2\circ J\quad\text{on }B,
 \quad f\in\Egam.
\]
If $x_k\in B$ converges to $x\in\Mji1$, then no subsequence of $J(x_k)$ can converge to a point of $\p M_2$.
\end{lem}
\begin{proof}
Suppose, to the contrary, that $y_k=J(x_k)\to y\in\p M_2$. Proposition~\ref{prop:operator_recovery} gives one constant $\kappa>0$ such that
\begin{equation}\label{eq:escape_metric_relation}
 J^*g_2=\kappa c^{p-2}g_1\quad\text{on }B.
\end{equation}
We consider the observed and unobserved boundary separately.

If $y\in\Gamma$, choose $h\in\Egam$ with $h(y)=1$. Then
\[
 c(x_k)=\frac{u_h^1(x_k)}{u_h^2(y_k)}
 \longrightarrow c_*:=u_h^1(x),
\]
since $u_h^2(y_k)\to1$. Taking the limit in the solution identity gives
\[
 u_f^1(x)=c_*f(y)\quad\text{for all }f\in\Egam.
\]
The scalar $c_*$ cannot be zero, since $P_1(x)\ne0$. Choose a nonempty open $\Sigma\Subset\Gamma\setminus\{y\}$. The prescribed-jet proposition with boundary values supported in $\Sigma$ supplies $f$ with $u_f^1(x)=1$, while $f(y)=0$, a contradiction.

Suppose now that $y\notin\Gamma$. Lemma~\ref{lem:remote_normal} gives $h\in\Egam$ such that
\[
 \p_{\nu_{g_2}}u_h^2(y)\ne0.
\]
The datum $h$ is zero near $y$. Write $t$ for inward $g_2$-normal distance near $y$. We have
\[
 u_h^2=t b_h,\quad b_h(y,0)\ne0.
\]
For any fixed $f\in\Egam$, its datum also vanishes on some neighborhood of $y$, possibly depending on $f$. On that neighborhood $u_f^2=t b_f$, with $b_f$ smooth. It follows that
\begin{equation}\label{eq:boundary_quotient_limit}
 \frac{u_f^2(y_k)}{u_h^2(y_k)}
 \longrightarrow\frac{b_f(y,0)}{b_h(y,0)}.
\end{equation}
For all sufficiently large $k$, cancellation of the multiplier gives
\[
 u_f^1(x_k)=u_h^1(x_k)\frac{u_f^2(y_k)}{u_h^2(y_k)}.
\]
If $u_h^1(x)=0$, equation~\eqref{eq:boundary_quotient_limit} would imply $u_f^1(x)=0$ for every $f$, again contradicting $P_1(x)\ne0$. Therefore $u_h^1(x)\ne0$.

The metric
\[
 \widehat g_1=|u_h^1|^{p-2}g_1
\]
is smooth and nondegenerate on a neighborhood of the interior point $x$. All of its sectional curvatures are uniformly bounded on a smaller neighborhood. The metric
\[
 \widehat g_2=|u_h^2|^{p-2}g_2
\]
is smooth in a small interior neighborhood of $y$, and \eqref{eq:escape_metric_relation} gives
\[
 J^*\widehat g_2=\kappa\widehat g_1
\]
wherever these metrics are used. Thus, for every two-plane $\Pi\subset T_{y_k}M_2$ and all sufficiently large $k$,
\[
 |K_{\widehat g_2}(\Pi)|
 =\kappa^{-1}|K_{\widehat g_1}((DJ_{x_k})^{-1}\Pi)|
 \leq C\kappa^{-1}.
\]
To obtain a contradiction, put $u=|u_h^2|$ and $\beta=2/(n-2)$. Since $u_h^2=t b_h$ and $b_h(y,0)\ne0$, the function $u$ is smooth up to the boundary on a smaller neighborhood, is positive on its interior, and has $du(y)\ne0$. At each $y_k$, choose a $g_2$-orthonormal pair $e_1,e_2\in\ker du(y_k)$ and let $\Pi_k$ be its span. This is possible because $n\geq3$. The sectional curvature formula for $\widehat g_2=u^{2\beta}g_2$, with the convention that the sphere has positive curvature, gives
\begin{align}\label{eq:boundary_curvature_contradiction}
 u^{2+2\beta}K_{\widehat g_2}(\Pi_k)
 ={}&u^2K_{g_2}(\Pi_k)
 -\beta u\bigl(\Hess_{g_2}u(e_1,e_1)+\Hess_{g_2}u(e_2,e_2)\bigr)-\beta^2|du|_{g_2}^2.
\end{align}
All quantities are evaluated at $y_k$. The curvature of $g_2$ and the Hessian of $u$ are bounded uniformly over these orthonormal pairs. The right-hand side therefore tends to $-\beta^2|du(y)|_{g_2}^2<0$. The left-hand side tends to zero by the curvature bound above and $u(y_k)\to0$. This contradiction excludes the unobserved boundary limit.
\end{proof}

\begin{rmk}
The last proof does not assume that $c$ extends to the limiting boundary point. In the unobserved case it first uses the quotient limit to find a nonvanishing solution on the source side, and only then compares the two nondegenerate metrics in the interior. Thus it also covers the possibility that $c(x_k)$ tends to infinity.
\end{rmk}

\subsection{Continuation in \texorpdfstring{$Z$}{Z}-coordinates}
\begin{thm}[Global determination of scaled solutions]\label{thm:global_scaled}
Let $(M_j,g_j,q_j)$, $j=1,2$, satisfy the smooth assumptions of Section~\ref{sec_poisson_embed}, and suppose in addition that the interiors, metrics, and potentials are real analytic. Let $B_0\subset\Mji1$ be a nonempty connected open set. Assume that there are a smooth diffeomorphism $F:B_0\to F(B_0)\subset\Mji2$ and a positive smooth function $c_0$ such that
\begin{equation}\label{eq:global_seed}
 u_f^1=c_0\,u_f^2\circ F
 \quad\text{on }B_0,\quad f\in\Egam.
\end{equation}
Then $F$ and $c_0$ have unique extensions to an analytic diffeomorphism
\[
 J:\Mji1\longrightarrow\Mji2
\]
and a positive analytic function $c$ on $\Mji1$ satisfying
\begin{equation}\label{eq:global_solution_identity}
 u_f^1=c\,u_f^2\circ J
 \quad\text{on }\Mji1,\quad f\in\Egam.
\end{equation}
Consequently,
\begin{equation}\label{eq:global_scaled_images}
 P_1(\Mji1)=\bigl\{(c\circ J^{-1})(y)P_2(y):y\in\Mji2\bigr\}.
\end{equation}
\end{thm}
\begin{proof}
Lemma~\ref{J_is_local_dif} shows first that the initial $F$ and $c_0$ are analytic. Consider all connected open subsets of $\Mji1$ containing $B_0$ on which they extend with \eqref{eq:global_solution_identity}. Two such extensions agree on every part of their overlap: if $c(x)P_2(J(x))=\wt c(x)P_2(\wt J(x))$, Proposition~\ref{prop:projective_separation} identifies both the point and the multiplier. Hence they glue to a largest connected open set $B\subset\Mji1$, an analytic positive $c$ on $B$, and an analytic map $J$ satisfying the solution identity. The map is a diffeomorphism onto its image by Lemma~\ref{J_is_local_dif}.

We show that $B$ is closed relative to $\Mji1$. Let $x_k\in B$ and $x_k\to x\in\Mji1$. By compactness of $M_2$, after passing to a subsequence we may assume
\[
 y_k=J(x_k)\longrightarrow y\in M_2.
\]
Lemma~\ref{lem:no_boundary_escape} gives $y\in\Mji2$.

Choose $f_0\in\Egam$ such that $u_{f_0}^2(y)=1$. For large $k$,
\[
 c(x_k)=\frac{u_{f_0}^1(x_k)}{u_{f_0}^2(y_k)}
 \longrightarrow a:=u_{f_0}^1(x).
\]
The limit is nonnegative. Taking limits for all other data gives
\begin{equation}\label{eq:limit_scaled_identity}
 u_f^1(x)=a\,u_f^2(y)\quad(f\in\Egam).
\end{equation}
If $a=0$, all solution values at $x$ vanish, which is impossible. Therefore $a>0$.

By Lemma~\ref{lem:Z_existence}, choose numerators $f_1,\ldots,f_n$ so that
\[
 U=\frac{(u_{f_1}^2,\ldots,u_{f_n}^2)}{u_{f_0}^2}
\]
is an analytic coordinate system on a neighborhood $\Omega_2$ of $y$. Shrink it so that $u_{f_0}^2>0$ there. Since $u_{f_0}^1(x)=a>0$, the map
\[
 V=\frac{(u_{f_1}^1,\ldots,u_{f_n}^1)}{u_{f_0}^1}
\]
is analytic on a neighborhood of $x$, although it is not yet known to be a coordinate system. Equation~\eqref{eq:limit_scaled_identity} gives $V(x)=U(y)$. Choose a connected neighborhood $\Omega_1$ of $x$ so small that $u_{f_0}^1>0$ there and $V(\Omega_1)\subset U(\Omega_2)$, and define
\begin{equation}\label{eq:extension_F}
 \wt J=U^{-1}\circ V\quad\text{on }\Omega_1.
\end{equation}
This is an analytic map into $\Omega_2$.

For some sufficiently large $k$, both $x_k\in\Omega_1$ and $J(x_k)\in\Omega_2$. Continuity of $J$ on $B$ gives a nonempty open neighborhood $W\subset B\cap\Omega_1$ of $x_k$ with $J(W)\subset\Omega_2$. On $W$ the solution identity implies $V=U\circ J$, hence $\wt J=J$. For every $f\in\Egam$ we consequently have
\[
 \frac{u_f^1}{u_{f_0}^1}
 =\left(\frac{u_f^2}{u_{f_0}^2}\right)\circ\wt J
 \quad\text{on }W.
\]
Both sides are real analytic on the connected set $\Omega_1$. The identity theorem therefore gives
\begin{equation}\label{eq:continued_quotient_identity}
 \frac{u_f^1}{u_{f_0}^1}
 =\left(\frac{u_f^2}{u_{f_0}^2}\right)\circ\wt J
 \quad\text{on all of }\Omega_1
\end{equation}
for every $f$. No analyticity of an unknown multiplier is used in this step.

Set
\begin{equation}\label{eq:continued_multiplier}
 \wt c=\frac{u_{f_0}^1}{u_{f_0}^2\circ\wt J}>0
 \quad\text{on }\Omega_1.
\end{equation}
Equations~\eqref{eq:continued_quotient_identity}--\eqref{eq:continued_multiplier} give
\[
 u_f^1=\wt c\,u_f^2\circ\wt J\quad\text{on }\Omega_1
\]
for every $f$. Lemma~\ref{J_is_local_dif} now applies to the continuous map $\wt J$ and proves that it is a diffeomorphism onto its image. Equivalently, differentiating \eqref{eq:continued_quotient_identity} and using \eqref{eq:quotient_differentials} shows that its differential has no kernel. On every part of $\Omega_1\cap B$, point separation implies $\wt J=J$ and $\wt c=c$. We can therefore glue the maps and multipliers to $B\cup\Omega_1$. This set is connected because $W$ is nonempty. Maximality of $B$ forces $x\in B$.

It follows that $B$ is relatively closed. Since the interior of a connected manifold with boundary is connected and $B$ is nonempty and open, $B=\Mji1$. We have constructed an injective analytic local diffeomorphism from the whole first interior into the second.

It remains to prove surjectivity. On $F(B_0)$ the inverse local map has the solution relation
\[
 u_f^2=\left(\frac1{c_0}\circ F^{-1}\right)u_f^1\circ F^{-1}.
\]
Apply the construction just proved with the two manifolds interchanged. It gives an analytic map $K:\Mji2\to\Mji1$ and a positive analytic $d$ such that
\[
 u_f^2=d\,u_f^1\circ K
\]
for every $f$. For each $x\in\Mji1$,
\[
 P_1(x)=c(x)d(J(x))P_1(K(J(x))).
\]
Point separation gives $K(J(x))=x$ and $c(x)d(J(x))=1$. The same argument on $M_2$ gives $J(K(y))=y$. Thus $J$ is surjective and $K=J^{-1}$. Uniqueness of $J$ and $c$ follows from the same separation statement. Finally, \eqref{eq:global_scaled_images} follows from \eqref{eq:global_solution_identity} by writing $y=J(x)$.
\end{proof}

\subsection{Proof of the main theorem}
\begin{proof}[Proof of Theorem~\ref{thm:main}]
Choose $x_0\in\Gamma$ and apply Lemma~\ref{lemma_determination_near_the_boundary}. Restrict its local identification to a connected nonempty interior neighborhood $B_0$. Theorem~\ref{thm:global_scaled} extends it to an analytic diffeomorphism $J:\Mji1\to\Mji2$ and a positive analytic $c$ satisfying \eqref{eq:global_solution_identity}.

Proposition~\ref{prop:operator_recovery} gives \eqref{eq:local_coefficient_recovery} on the connected interior. On the initial set $B_0$, Lemma~\ref{lemma_determination_near_the_boundary} already gives $J^*g_2=c^{p-2}g_1$. Therefore the constant in that proposition is $\kappa=1$. We obtain \eqref{eq:main_gauge} throughout the interior and the operator identity
\begin{equation}\label{eq:final_intertwining}
 L_{g_1,q_1}(c\,J^*v)=c^{p-1}J^*(L_{g_2,q_2}v)
\end{equation}
for every locally defined smooth $v$.

We next attach the observed boundary. Fix any $z\in\Gamma$ and apply Lemma~\ref{lemma_determination_near_the_boundary} at $z$, obtaining a local pair $(F_z,c_z)$. On the interior part of its domain, both this pair and $(J,c)$ satisfy the same scaled identity for all data. Proposition~\ref{prop:projective_separation} therefore gives
\[
 F_z=J,\quad c_z=c
\]
there. These local pairs extend $J$ and $c$ analytically to every point of $\Gamma$. They agree on overlaps because they already agree in the interior. Their boundary values give \eqref{eq:boundary_normalization}; their local inverses extend $J^{-1}$ to the boundary. We have thus obtained an analytic diffeomorphism $M_1^\Gamma\to M_2^\Gamma$, and \eqref{eq:main_gauge} and \eqref{eq:main_solution_relation} extend to $\Gamma$ as well.

Uniqueness for pairs satisfying the solution relation follows directly from Proposition~\ref{prop:projective_separation} and nonvanishing of $P_j$. In fact, the gauge relations and boundary normalization already force that solution relation. To see this, let $(\wt J,\wt c)$ be another pair satisfying \eqref{eq:boundary_normalization}--\eqref{eq:main_gauge}. For any $f$, covariance shows that $\wt c\,u_f^2\circ\wt J$ solves $L_{g_1,q_1}u=0$ in the interior. On $\Gamma$ it has the same Dirichlet and Neumann values as $u_f^1$: the metric identity and $\wt c=1$ imply $D\wt J(\nu_{g_1})=\nu_{g_2}$ there, and the normal derivative of $\wt c$ is zero. Equality of the local DN maps and boundary unique continuation, followed by interior unique continuation, give
\[
 u_f^1=\wt c\,u_f^2\circ\wt J
\]
on the whole interior. Point separation now identifies the two pairs. This concludes the proof.
\end{proof}

\begin{proof}[Proof of Corollary~\ref{cor:full_data}]
Take $\Gamma=\p M$ in Theorem~\ref{thm:main}, so that $M_j^\Gamma=M_j$. Conversely, a pair satisfying the stated global gauge and boundary conditions transforms each solution on $M_2$ to a solution on $M_1$ with the same full Dirichlet datum, by \eqref{eq:covariance_intro}. Uniqueness of the Dirichlet problem and the normal-derivative calculation in Proposition~\ref{hidden_conformal_invariance} give equality of the full DN maps.
\end{proof}

\appendix
\section{Runge approximation and construction of prescribed jets}\label{runge_apprx_sec}
This appendix proves the approximation and jet results used above. It is independent of the global continuation argument. We assume throughout that \(M\) is compact and connected with smooth boundary, that \(g\) and the real-valued potential \(q\) are smooth, and that the Dirichlet realization of \(L=\Delta_g+q\) is invertible. Similar arguments were used in \cite[Appendix B]{lassas2019conformal} and \cite[Appendix A]{lassas2018poisson}. Here we generalize them to the present setting.


\subsection{Local solutions with a prescribed second-order jet}
\begin{lem}\label{prescribed_sols}
Let $x\in\Mi$, let $a\in\R$, let $\xi\in T_x^*M$, and let $H$ be a symmetric covariant two-tensor at $x$ with $\tr_gH=q(x)a$. There is a smooth solution $v$ of $Lv=0$ on a neighborhood of $x$ such that
\[
 v(x)=a,\quad dv(x)=\xi,\quad\Hess_gv|_x=H.
\]
In particular, there is a local positive solution with value one at $x$, and there are local solutions with value zero and arbitrary first derivative at $x$.
\end{lem}
\begin{proof}
Choose $g$-normal coordinates centered at $x$, so that $x$ corresponds to $0$, $g^{ij}(0)=\delta^{ij}$, and the Christoffel symbols vanish at $0$. In these coordinates
\[
 Lv=-g^{ij}(z)\p_{ij}v+b^k(z)\p_kv+q(z)v,
 \quad b^k(0)=0.
\]
Let
\[
 \mathcal P_2=
 \left\{a'+\eta_jz^j+\tfrac12S_{ij}z^iz^j:
                 S=S^T,\ \tr S=0\right\}
\]
be the finite-dimensional space of Euclidean harmonic polynomials of degree at most two. The map
\[
 P\longmapsto\bigl(P(0),dP(0),(D^2P(0))_{\tf}\bigr)
\]
is an isomorphism from $\mathcal P_2$ onto
$E=\R\oplus(\R^n)^*\oplus\operatorname{Sym}^2_{\tf}((\R^n)^*)$.

For small $\varepsilon>0$, rescale a coordinate ball to the unit ball and define
\begin{equation}\label{eq:rescaled_L}
 L_\varepsilon w
 =-g^{ij}(\varepsilon z)\p_{ij}w
   +\varepsilon b^k(\varepsilon z)\p_kw
   +\varepsilon^2q(\varepsilon z)w.
\end{equation}
Fix $0<\alpha<1$. As an operator from
$C^{2,\alpha}(\overline{B_1})\cap\{w|_{\p B_1}=0\}$
to $C^{0,\alpha}(\overline{B_1})$, $L_\varepsilon$ converges in operator norm to the Euclidean Dirichlet operator $A_0=-\sum_i\p_{ii}$, which is invertible. The Neumann-series perturbation argument therefore gives invertibility for all sufficiently small $\varepsilon$, with a uniformly bounded inverse. In particular, for each $P\in\mathcal P_2$ there is a unique solution
\[
 L_\varepsilon w_\varepsilon^P=0\text{ in }B_1,
 \quad w_\varepsilon^P=P\text{ on }\p B_1.
\]
The correction satisfies
\[
 L_\varepsilon(w_\varepsilon^P-P)=-L_\varepsilon P,
 \quad (w_\varepsilon^P-P)|_{\p B_1}=0.
\]
Since $A_0P=0$, uniform elliptic estimates give
\[
 \norm{w_\varepsilon^P-P}_{C^{2,\alpha}(\overline{B_1})}
 \leq C\norm{(L_\varepsilon-A_0)P}_{C^{0,\alpha}(\overline{B_1})}
 \longrightarrow0.
\]
The convergence is uniform for $P$ in the unit sphere of the finite-dimensional space $\mathcal P_2$.

It follows that the linear maps
\[
 S_\varepsilon:\mathcal P_2\longrightarrow E,
 \quad S_\varepsilon P
 =\bigl(w_\varepsilon^P(0),dw_\varepsilon^P(0),
                         (D^2w_\varepsilon^P(0))_{\tf}\bigr)
\]
converge to the isomorphism given by the polynomial jets. Hence $S_\varepsilon$ is an isomorphism for all sufficiently small fixed $\varepsilon$.

Write $H_0=H-(q(0)a/n)g_0$, so $\tr H_0=0$. Fix one such $\varepsilon$ and choose $P$ with
\[
 S_\varepsilon P=(a,\varepsilon\xi,\varepsilon^2H_0).
\]
Returning to the original coordinates, set
\[
 v(z)=w_\varepsilon^P(z/\varepsilon).
\]
Then $Lv=0$ on the ball of radius $\varepsilon$, $v(0)=a$, $dv(0)=\xi$, and the trace-free part of $\Hess_gv(0)$ is $H_0$. Evaluation of \eqref{eq:rescaled_L} at $0$ gives
\[
 \tr D^2w_\varepsilon^P(0)=\varepsilon^2q(0)a,
\]
because $b(0)=0$. Therefore the full Hessian of $v$ is $H$. Smooth elliptic regularity gives a smooth solution on a neighborhood of $0$. Taking $a=1$ and then shrinking the neighborhood gives a positive solution by continuity. Taking $a=0$ allows arbitrary $\xi$.
\end{proof}

\subsection{Exact global jets at finitely many points}
\begin{proof}[Proof of Proposition~\ref{runge_prescribed_jet}]
Choose pairwise disjoint coordinate balls around $x_1,\ldots,x_N$, with their closures in the interior. By Lemma~\ref{prescribed_sols}, in each ball, after shrinking it, there is a local solution with the prescribed compatible jet. Choose smaller smooth balls $O_\ell$ whose closures lie in these solution neighborhoods, and put $O=\bigcup_\ell O_\ell$. The local solutions define a single smooth solution $v$ on a neighborhood of $\overline O$, since the components are disjoint.

The balls can be chosen so that $M\setminus\overline O$ is connected. To see this directly, take disjoint slightly larger coordinate balls in the interior. Any path joining two points outside the smaller balls can be modified inside each larger ball to avoid the corresponding smaller closed ball: the annular region is path connected in dimension at least two. Performing this modification for the finitely many balls proves the claim. A connected manifold with boundary has connected interior, by pushing paths slightly inward in a collar.

Apply Proposition~\ref{prop_runge_approximation_cauchy}. The resulting global solutions converge to $v$ in $C^2$ on smaller compact neighborhoods of all the points. Their simultaneous second-order jets therefore converge to the prescribed jets.

To replace approximation by equality, let
\[
 E_\ell=
 \{(a,\xi,H):\tr_gH=q(x_\ell)a\},
 \quad E=\prod_{\ell=1}^N E_\ell,
\]
and consider the linear map
\[
 T:\Egam\longrightarrow E,
 \quad
 Tf=\bigl(u_f(x_\ell),du_f(x_\ell),\Hess_gu_f|_{x_\ell}\bigr)_{\ell=1}^N.
\]
The equation shows that the range lies in $E$. The preceding approximation shows that this range is dense in $E$. It is a linear subspace of a finite-dimensional space, and is therefore closed. Thus $T$ is onto, proving exact realization of all the prescribed jets.
\end{proof}

\begin{rmk}
Only interior second-order jets are needed for the coefficient calculation and continuation argument. The boundary $Z$-coordinates in Lemma~\ref{lem:Z_existence} are constructed directly from boundary data and Lemma~\ref{lem:remote_normal}; no extension of the manifold or assumption about the Dirichlet spectrum on an enlarged manifold is required.
\end{rmk}

\section*{Acknowledgments}
C. I. C. is partially supported by the National Science and Technology Council (NSTC) of Taiwan under grant 113-2115-M-A49-018-MY3.

T. L. and W. T. are partly supported by the Research Council of Finland through the Centre of Excellence in Inverse Modelling and Imaging and the FAME Flagship, grants 353091 and 359208.

Y.-H. L. is partially supported by the National Science and Technology Council (NSTC) of Taiwan under grant 113-2115-M-A49-017-MY3. Y.-H. L. also acknowledges financial support from the Alexander von Humboldt Foundation through the Henriette Herz Scouting Programme during his research stay at Universit\"at Duisburg-Essen.

AI tools were used for mathematical discussions and for assistance in revising parts of the arguments and exposition.

\bibliographystyle{alpha}
\bibliography{ref}
\end{document}